\documentclass[reqno, 12pt]{amsart}
\pdfoutput=1
\makeatletter
\let\origsection=\section \def\section{\@ifstar{\origsection*}{\mysection}}
\def\mysection{\@startsection{section}{1}\z@{.7\linespacing\@plus\linespacing}{.5\linespacing}{\normalfont\scshape\centering\S}}
\makeatother

\usepackage{amsmath,amssymb,amsthm}
\usepackage{mathrsfs}
\usepackage{mathabx}\changenotsign
\usepackage{dsfont}

\usepackage[normalem]{ulem}

\usepackage[dvipsnames]{xcolor}
\usepackage[backref]{hyperref}
\hypersetup{
    colorlinks,
    linkcolor={red!60!black},
    citecolor={green!60!black},
    urlcolor={blue!60!black}
}

\usepackage{graphicx}

\usepackage[open,openlevel=2,atend]{bookmark}

\usepackage[abbrev,msc-links,backrefs]{amsrefs}
\usepackage{doi}

\renewcommand{\PrintDOI}[1]{\doi{#1}}

\usepackage[T1]{fontenc}
\usepackage{lmodern}
\usepackage[babel]{microtype}
\usepackage[english]{babel}

\usepackage{geometry}
\numberwithin{equation}{section}
\numberwithin{figure}{section}

\usepackage{enumitem}

\usepackage{pgf,tikz}

\usepackage{listings}
\usepackage{xcolor}

\definecolor{codegreen}{rgb}{0,0.6,0}
\definecolor{codelightgray}{gray}{0.9}

\lstdefinestyle{Maplestyle}{
    backgroundcolor=\color{codelightgray},
    commentstyle=\color{codegreen},
    keywordstyle=\color{black}\bfseries,
    stringstyle=\color{blue},
    basicstyle=\ttfamily\footnotesize,
    breakatwhitespace=false,
    breaklines=true,
    captionpos=b,
    keepspaces=true,
    showspaces=false,
    showstringspaces=false,
    showtabs=false,
    tabsize=4
}

\lstdefinelanguage{Maple}%
{morekeywords={and,assuming,break,by,catch,description,do,done,%
elif,else,end,error,export,fi,finally,for,from,global,if,%
implies,in,intersect,local,minus,mod,module,next,not,od,%
option,options,or,proc,quit,read,return,save,stop,subset,then,%
to, try,union,use,uses,while,xor},%
sensitive=true,%
morecomment=[l]\#,%
morestring=[b]",%
morestring=[d]"%
}[keywords,comments,strings]%

\usepackage{array,multirow,colortbl}
\usepackage{caption} 

\makeatletter
\def\greek#1{\expandafter\@greek\csname c@#1\endcsname}
\def\Greek#1{\expandafter\@Greek\csname c@#1\endcsname}
\def\@greek#1{\ifcase#1
	\or $\alpha$%
	\or $\beta$%
	\or $\gamma$%
	\or $\delta$%
	\or $\epsilon$%
	\or $\zeta$%
	\or $\eta$%
	\or $\theta$%
	\or $\iota$%
	\or $\kappa$%
	\or $\lambda$%
	\or $\mu$%
	\or $\nu$%
	\or $\xi$%
	\or $o$%
	\or $\pi$%
	\or $\rho$%
	\or $\sigma$%
	\or $\tau$%
	\or $\upsilon$%
	\or $\phi$%
	\or $\chi$%
	\or $\psi$%
	\or $\omega$%
\fi}
\def\@Greek#1{\ifcase#1
	\or $\mathrm{A}$%
	\or $\mathrm{B}$%
	\or $\Gamma$%
	\or $\Delta$%
	\or $\mathrm{E}$%
	\or $\mathrm{Z}$%
	\or $\mathrm{H}$%
	\or $\Theta$%
	\or $\mathrm{I}$%
	\or $\mathrm{K}$%
	\or $\Lambda$%
	\or $\mathrm{M}$%
	\or $\mathrm{N}$%
	\or $\Xi$%
	\or $\mathrm{O}$%
	\or $\Pi$%
	\or $\mathrm{P}$%
	\or $\Sigma$%
	\or $\mathrm{T}$%
	\or $\mathrm{Y}$%
	\or $\Phi$%
	\or $\mathrm{X}$%
	\or $\Psi$%
	\or $\Omega$%
\fi}

\AddEnumerateCounter{\greek}{\@greek}{24}
\AddEnumerateCounter{\Greek}{\@Greek}{12}

\makeatother

\let\polishlcross=\l
\def\l{\ifmmode\ell\else\polishlcross\fi}

\def\paragraph#1{%
  \noindent\textbf{#1.}\enspace}

\let\setminus=\smallsetminus

\makeatletter
\def\moverlay{\mathpalette\mov@rlay}
\def\mov@rlay#1#2{\leavevmode\vtop{   \baselineskip\z@skip \lineskiplimit-\maxdimen
   \ialign{\hfil$\m@th#1##$\hfil\cr#2\crcr}}}
\newcommand{\charfusion}[3][\mathord]{
    #1{\ifx#1\mathop\vphantom{#2}\fi
        \mathpalette\mov@rlay{#2\cr#3}
      }
    \ifx#1\mathop\expandafter\displaylimits\fi}
\makeatother

\DeclareFontFamily{U}  {MnSymbolC}{}
\DeclareSymbolFont{MnSyC}         {U}  {MnSymbolC}{m}{n}
\DeclareFontShape{U}{MnSymbolC}{m}{n}{
    <-6>  MnSymbolC5
   <6-7>  MnSymbolC6
   <7-8>  MnSymbolC7
   <8-9>  MnSymbolC8
   <9-10> MnSymbolC9
  <10-12> MnSymbolC10
  <12->   MnSymbolC12}{}
\DeclareMathSymbol{\powerset}{\mathord}{MnSyC}{180}

\usepackage{tikz}
\usetikzlibrary{calc,decorations.pathmorphing}
\pgfdeclarelayer{background}
\pgfdeclarelayer{foreground}
\pgfdeclarelayer{front}
\pgfsetlayers{background,main,foreground,front}

\let\epsilon=\varepsilon
\let\eps=\epsilon
\let\rho=\varrho
\let\theta=\vartheta

\let\E=\EE

\def\PP{{\mathds P}}

\newcommand{\bin}{\mathrm{Bin}}

\theoremstyle{plain}
\newtheorem{thm}{Theorem}[section]

\newtheorem{prop}[thm]{Proposition}

\newtheorem{fact}[thm]{Fact}

\newtheorem{lemma}[thm]{Lemma}

\theoremstyle{definition}
\newtheorem{rem}[thm]{Remark}

\newtheorem{exmp}[thm]{Example}

\newtheorem{prob}[thm]{Problem}

\usepackage{accents}

\let\phi=\varphi

\begin{document}

\title{Long twins in random words}

\author{Andrzej Dudek}
\address{Department of Mathematics, Western Michigan University, Kalamazoo, MI, USA}
\email{\tt andrzej.dudek@wmich.edu}
\thanks{The first author was supported in part by Simons Foundation Grant \#522400.}

\author{Jaros\l aw Grytczuk}
\address{Faculty of Mathematics and Information Science, Warsaw University of Technology, Warsaw, Poland}
\email{j.grytczuk@mini.pw.edu.pl}
\thanks{The second author was supported in part by Narodowe Centrum Nauki, grant 2020/37/B/ST1/03298.}

\author{Andrzej Ruci\'nski}
\address{Department of Discrete Mathematics, Adam Mickiewicz University, Pozna\'n, Poland}
\email{\tt rucinski@amu.edu.pl}
\thanks{The third author was supported in part by Narodowe Centrum Nauki, grant 2018/29/B/ST1/00426}

\begin{abstract}\emph{Twins} in a finite word are formed by a pair of identical subwords placed at disjoint sets of positions. We investigate the maximum length of twins in \emph{a random} word over a $k$-letter alphabet. The obtained lower bounds for small values of $k$ significantly improve the best estimates known in the deterministic case.

Bukh and Zhou in 2016 showed that every ternary word of length $n$  contains twins of length at least $0.34n$. Our main result states that in a random ternary word of length $n$, with high probability, one can find twins of length at least $0.41n$. In the general case of alphabets of size $k\geq 3$ we obtain analogous lower bounds  of the form $\frac{1.64}{k+1}n$ which are better than the known deterministic bounds for $k\leq 354$. In addition, we present similar results for \emph{multiple} twins in random words. \end{abstract}

\maketitle


\setcounter{footnote}{1}

\section{Introduction}
Looking for twin objects has had a long tradition in all branches of mathematics. For discrete structures a first question of this kind is attributed  to Ulam (see, e.g.,~\cite{ChungEGUY}) who
proposed to measure  similarity of two graphs in terms of their edge decompositions  into pairwise isomorphic subgraphs. If the number of subgraphs in a decomposition is small, then there must be a large pair among them, which leads to the following general question: given a combinatorial structure (or a pair of structures), how large disjoint isomorphic substructures does it (do they) contain?



This problem has been studied in various forms: for graphs (\cite{AlonCaroKrasikov}, \cite{LeeLohSudakov}), words (\cite{APP}, \cite{BukhZ}), and permutations (\cite{BukhR}, \cite{DGR-Electronic}, \cite{DGR-Integers}, \cite{DGR-Annals}). Lots of interesting results and challenging open problems can be found in these papers and their references. Some of them are collected in a survey \cite{Axenovich}.

\subsection{Twins in words: history and background}\label{Section Twins History}

In this paper we study the problem of twins in random words. To put our work in a broader context, we briefly recall what is known on twins in words in general.

The first who introduced and studied this problem were Axenovich, Person, and Puzynina~\cite{APP}. We say that a word $w=w_1w_2\cdots w_n\in A^n$, $|A|<\infty$, contains \emph{twins} of \emph{length $t$} if there exist \emph{disjoint} subsets $\{i_1<i_2<\cdots<i_t\}$ and $\{j_1<j_2<\cdots<j_t\}$ of indices such that $w_{i_1}w_{i_2}\cdots w_{i_t}=w_{j_1}w_{j_2}\cdots w_{j_t}$. Let $t(w)$ denote the maximum length of twins in $w$ and let $t(n,k)$ denote the minimum of $t(w)$ over all words $w$ of length $n$ over a $k$-element alphabet $A$.
They proved in~\cite{APP} that $t(n,2)=n/2-o(n)$. In other words,  every finite binary word can be split into a pair of identical subwords, up to an asymptotically negligible remainder. To get this striking result they developed a regularity lemma for words -- an interesting tool with a great potential for further applications (see, e.g.,~\cite{HanKP-S}).

At present, it is not known whether the same is true for  larger alphabets. By considering the sub-word formed by two most frequent letters, the above mentioned estimate $t(n,2)=n/2-o(n)$  from~\cite{APP} implies that $t(n,k)\geq n/k -o(n)$. For $k\geq3$, this was slightly improved by Bukh and Zhou~\cite{BukhZ} to
\begin{equation}\label{Bukh-Zhou t_k(n)1}
t(n,k)\geq 1.02\cdot\frac{n}{k} -o(n)
\end{equation} and to
\begin{equation}\label{Bukh-Zhou t_k(n)2}
t(n,k)\geq \left(\frac{k}{81}\right)^{1/3}\cdot \frac{n}{k}-(k/3)^{1/3},
\end{equation}
 the latter being a (much) better estimate for larger $k$. In~\cite{APP} and~\cite{BukhZ} there are also some upper bounds on $t(n,k)$ valid for $k\ge4$. In particular, $t(n,4)\leq0.4932n$.

Twins in \emph{random} words were not much studied before, although the proofs of the upper bounds in~\cite{APP} and~\cite{BukhZ} were obtained via the probabilistic method. The only work addressing this issue directly is a recent paper by He, Huang, Nam, and Thaper \cite{HeHNT}, where it is proved that, with high probability, binary words of length~$n$ contain twins of size $n/2-\omega\sqrt{n}$, for any function $\omega=\omega(n)$ tending to infinity with $n$. Moreover, the authors of \cite{HeHNT} formulate a striking conjecture  that almost every binary word with even numbers of ones and zeros is a perfect pair of twins. If true, this would imply that, with high probability, a random binary word of an \emph{even} length $n$ contains twins of size at least $n/2-1$ (when the number of ones is odd, this is the best one can get).


\subsection{Our results}
In this paper we study twins in a \emph{random word} $W_k(n)$ obtained by drawing one with probability $k^{-n}$ out of all $k$-ary words of length $n$. Equivalently, one could toss a $k$-sided fair die, independently, $n$ times. Either way, this is an equiprobable space.

Our main result improves the  lower bound  \eqref{Bukh-Zhou t_k(n)1} for  all but very few ternary ($k=3)$ words of length $n$.

\begin{thm}\label{thm:0.4} With probability $1-e^{-\Omega(n/(\log n))}$,
$$t(W_3(n))\ge 0.411 n.$$
\end{thm}

The proof involves computer-assisted calculations (see Appendix~\ref{sec:maple}).  It seems plausible that substantially stronger computational devices could bring further improvements of the bound. Moreover, by using classical methods (without computers) we provide a slightly worse bound, namely, $0.375n$ (cf.~Example \ref{375} in Subsection \ref{apple}).

 For larger alphabets we have the following result.  We present it in a form which emphasizes the improvement over the deterministic bound \eqref{Bukh-Zhou t_k(n)1}. We say that an event $\mathcal E_n$ holds \emph{asymptotically almost surely (a.a.s.)} if $\Pr(\mathcal E_n)\to1$ as $n\to\infty$.

 \begin{thm}\label{main1} For each $k\ge3$ and large $n$, a.a.s.
\[
	t(W_k(n))\ge\frac{1.64k}{k+1}\cdot\frac{n}{k},
\]
\end{thm}
 \noindent This gives an improvement upon the (deterministic) estimates (\ref{Bukh-Zhou t_k(n)1}) and, for all $k\leq 354$, upon~(\ref{Bukh-Zhou t_k(n)2}) (see Table~\ref{table:r=2}).

\begin{table}
\begin{tabular}{ |>{\columncolor{codelightgray}}c|c | c | c | c | c | c | c | c |}
\hline
\rowcolor{codelightgray}
$k$ & 3 & 4 & 5 & 10 & 50 & 100 & 200 &  400\\ \hline
$\left(\frac{k}{81}\right)^{1/3}$ & 0.333  & 0.367  & 0.395  & 0.498  &  0.851 & 1.073   &  1.352 &  \textcolor{red}{1.703} \\ \hline
$\frac{1.64k}{k+1}$ & 1.230  & 1.312  & 1.367  & 1.491  &  1.608 & 1.624  & 1.632  &  1.636 \\ \hline
\end{tabular}
\caption{Comparing bound~\eqref{Bukh-Zhou t_k(n)2} of Bukh and Zhou~\cite{BukhZ} (for all words) with  Theorem~\ref{main1} (for almost all words).}
\label{table:r=2}
\end{table}

 The proof uses a tool called the Boosting Lemma (Lemma \ref{k2k+1}). It is stated in terms of a special model of random words which allows for iterative enhancement of the twin length, while adding a new letter to the alphabet. This new model assumes that the numbers of occurrences of letters are fixed in such a way that the model  is asymptotically equivalent to $W_k(n)$.

\subsection{Multiple twins}
We also consider a more general notion of \emph{multiple} twins. By \emph{$r$-twins} in a word $w$ we mean $r$ disjoint identical subwords of $w$. Let $t^{(r)}(w)$ be the maximum length of $r$-twins in $w$ and let $t^{(r)}(n,k)$ be the minimum of $t^{(r)}(w)$ over all words $w$ of length $n$ from a $k$-letter alphabet. By the results of~\cite{APP} and~\cite{BukhZ} we know, respectively, that $t^{(r)}(n,k)\sim n/r$ when $r\geq k$, and
\begin{equation}\label{BZr}
t^{(r)}(n,k)\geq C_r\cdot k^{1/\binom{2r-1}{r}}\cdot\frac{n}{k}-O(1),
\end{equation}
for $k>r\geq3$, where $C_r= \left(\frac{1}{2r-1}\right)^{1+1/\binom{2r-1}{r}}$ and the $O(1)$ term depends on $r$ and $k$ only.
Estimate \eqref{BZr} was only mentioned in the concluding remarks in \cite{BukhZ} and the explicit form of $C_r$ was not given. However, based on the proof of inequality \eqref{Bukh-Zhou t_k(n)2} therein, it is quite obvious how it should be computed.

For a random word $W_k(n)$ we get the following estimates which, again, for small $r$ and $k$  yield better lower bounds than \eqref{BZr} (see Table~\ref{table:r,k}). For $r,k\ge2$, let  $\Pi_{r,k}=\prod_{j=r+1}^k\frac{j^r}{j^r-1}$.

\begin{thm}\label{main2} For every $k>r\ge3$, a.a.s.
\[
	t^{(r)}(W_k(n))\ge\Pi_{r,k}\cdot\frac{n}{k}-o(n).
\]
\end{thm}

\begin{table}
\begin{tabular}{ |>{\columncolor{codelightgray}}c|c | c | c | c | c | c | c | c |}
\hline
\rowcolor{codelightgray}
$(r,k)$ & $(3,4)$ & $(3,10)$ & $(3,100)$ & $(3,1000)$ & $(3,10^{10})$ & $(4,10^{10})$  & $(4,10^{40})$ \\ \hline
$C_r\cdot k^{1/\binom{2r-1}{r}}$ & 0.196 & 0.214 & 1.036 & 0.340 & \textcolor{red}{1.703} & 0.261 & \textcolor{red}{1.878}\\ \hline
$\Pi_{r,k}$ & 1.016 & 1.036 & 1.041 & 1.041 & 1.041 & 1.003 & 1.003\\ \hline
\end{tabular}
\caption{Comparing bound~\eqref{BZr} of Bukh and Zhou~\cite{BukhZ} (for all words) with Theorem~\ref{main2} (for almost all words).}
\label{table:r,k}
\end{table}

\subsection{Organization} In the next section we prove Theorem \ref{thm:0.4}. Section \ref{genbou} begins with a short proof of the known bound $t^{(r)}(n,k)\sim n/r$, $r\ge k$, for random words. We do so for self-containment, as the proof in \cite{APP} (for \emph{all} words) is quite involved. The next subsection contains a standard proof of an asymptotic equivalence between $W_k(n)$ and another model of random words. Then comes the crucial Boosting Lemma, while the last subsection of Section \ref{genbou} brings applications of the Boosting Lemma, among them short proofs of Theorems \ref{main1} and \ref{main2}, the former utilizing also Theorem \ref{thm:0.4}.
The last section contains some remarks and open problems, while the Appendix presents a Maple code used for derivation of the data collected in Table~\ref{table:lambda}, as well as a proof of a technical estimate needed in Section \ref{boost}.

\section{Computer assisted bound}
In this section we prove Theorem \ref{thm:0.4}.

\begin{proof}[Proof of Theorem \ref{thm:0.4}]
Fix a positive integer $s$ and split the set of positions of  a ternary random word $W_3(n)$ into $m:=n/s$ segments of length~$s$ (we assume $s|n$).
For $1\le j\le m$ and $1\le t \le \lfloor \frac{s}{2} \rfloor$, let $X^j_t$ be the indicator random variable such that $X^j_t=1$ if the $j$-th segment of a random word $W_3(n)$ contains twins of length $t$ but not $t+1$; otherwise $X^j_t=0$. Furthermore, define $X_t:=\sum_{j=1}^m X^j_t$.

First we calculate the expected value of $X_t$. Let $\lambda_t$ count the number of ternary words of length $s$ with twins of length $t$ but not $t+1$. Clearly, $\sum_{t=1}^{\lfloor s/2 \rfloor} \lambda_t = 3^s$. In general, finding (or even tightly approximating) $\lambda_t$ does not seem to be an easy problem. However, for small  $s$ one can use a computer program to determine $\lambda_s$. In Table~\ref{table:lambda}  we present all values of $\lambda_t$ for all $6\le s\le 14$ and $1\le t\le\lfloor s/2\rfloor$ (see Appendix~\ref{sec:maple} for more details). Since $X_t$ has the binomial distribution $\bin(m,\lambda_t/3^s)$, its expectation is
\[
\E(X_t) = \frac{\lambda_t}{3^s} m = \frac{\lambda_t}{s3^s} n.
\]
Moreover, a standard application of the Chernoff inequality (see, e.g., ineq. (2.9) in~\cite{JLR}) together with the union bound yields that for all $1\le j\le m$ and $1\le t\le \lfloor \frac{s}{2} \rfloor$, $X_t$ is highly concentrated around its mean, i.e.,
\[
X_t = \frac{\lambda_t}{s3^s}(1+o(1)) n
\]
with probability $1-e^{-\Omega(n/(\log n))}$.

To finish the proof of Theorem~\ref{thm:0.4} we construct twins in $W_3(n)$ as follows. Let, for each $1\le j\le m$,  $A_j$ and $B_j$ be a pair of the longest twins in the $j$-th segment. Observe that concatenating $A_j$ and $B_j$ over all $m$ segments yields, a.a.s., twins $A_1A_2\dots A_m$ and $B_1B_2\dots B_m$ of length
$$\sum_{t=1}^{\lfloor s/2 \rfloor} tX_t=\sum_{t=1}^{\lfloor s/2 \rfloor}\frac{t\lambda_t}{s3^s}(1+o(1))n=\rho_s(1+o(1))n,$$
 where
\[
\rho_s:=\sum_{t=1}^{\lfloor s/2 \rfloor} \frac{t\lambda_t}{s3^s}.
\]
From the last column of Table~\ref{table:lambda} we compute that
$$\rho_{14}=\frac{4\cdot24\:\!894+5\cdot1\:\!312\:\!530+6\cdot 3\:\!196\:\!644+7\cdot248\:\!901}{14\cdot 3^{14}}=\frac{27\:\!584\:\!397}{66\:\!961\:\!566}>0.4119.$$
This completes the proof. \end{proof}


\begin{table}
\begin{tabular}{ |>{\columncolor{codelightgray}}c|c c c c c c c c c |} 
 \hline
\rowcolor{codelightgray}
$s$          & 6     & 7       & 8      & 9          & 10       & 11        & 12         & 13         & 14\\
\hline
$\lambda_1$ & 42   & 6     & \multicolumn{3}{|c}{} & \multicolumn{4}{c|}{\multirow{2}{*}{\textsf{Zero}}} \\ \cline{4-6}
$\lambda_2$ & 594 & 1\:\!086 & 822  & 288      & \multicolumn{1}{c|}{42} & \multicolumn{4}{c|}{}             \\ \cline{7-9}
$\lambda_3$ & 93   & 1\:\!095 & 5\:\!118 & 11\:\!010  & 10\:\!806 & 5\:\!292    & 1\:\!350     & 162 & \multicolumn{1}{|c|}{} \\ \cline{2-3} \cline{10-10}
$\lambda_4$ &  \multicolumn{2}{c|}{}                & 621   & 8\:\!385    & 43\:\!776 & 106\:\!032 & 123\:\!750 & 75\:\!810 & 24\:\!894\\ \cline{4-5}
$\lambda_5$ &   \multicolumn{4}{c|}{\multirow{2}{*}{}}          & 4\:\!425   & 65\:\!823   & 373\:\!638 & 992\:\!244 & 1\:\!312\:\!530\\ \cline{6-7}
$\lambda_6$ &   \multicolumn{6}{c|}{\textsf{Zero}}            & 32\:\!703   & 526\:\!107 & 3\:\!196\:\!644\\ \cline{8-9}
$\lambda_7$ & \multicolumn{8}{c|}{} & 248\:\!901 \\
 \hline
\end{tabular}
\caption{The exact values of $\lambda_t$ for all $6\le s\le 14$ and $1\le t \le \lfloor s/2 \rfloor$.}
\label{table:lambda}
\end{table}

\section{General bounds}\label{genbou}
\subsection{More twins than letters}\label{ratleastk} It was shown in~\cite{APP} that   for all $r\ge k$
$$t^{(r)}(n,k)\ge \frac nr-O\left(\left(\frac{\log\log n}{\log n}\right)^{1/4}\right).$$
For self-containment we insignificantly improve this lower bound for \emph{almost} all words. The new bound holds with probability so close to 1 that a passage to the  alternative model employed in Subsection \ref{apple} is still possible.

To this end, let us recall the main observation in~\cite{APP}: if there is a partition of a $k$-word~$w$ of length $n$ into $m$ segments $w_1w_2\cdots w_m$ of sizes $N:=n/m$ (we assume $m|n$) and in each segment each letter occurs at least $\mu$ times, then one can construct $r$-twins in $w$  via a natural interlacing. Namely,  Twin 1 consists of $\mu$ elements $a_1$ of $w_1$, $\mu$ elements $a_2$ of $w_2, \dots, \mu$ elements $a_k$ of $w_k$, followed by $\mu$ elements $a_{1}$ of $w_{r+1}$, $\mu$ elements $a_2$ of $w_{r+2}$, and so on, for as long as there are still at least $r-1$ segments ahead. Twin 2 follows the same pattern except that it begins with $\mu$ elements $a_1$ of $w_2$, that is, it is shifted by one segment with respect to Twin~1. Then Twin 3 begins with $\mu$ elements $a_1$ of $w_3$, etc.

 This way all $r$ twins are disjoint and they use together $k\mu$ elements from each segment except the first $k-1$ ones where the consumption gradually grows from $\mu$ to $(k-1)\mu$,   and, in the worst case, when $m=r-1\ (\!\!\!\mod{r})$, the last $r-1$ ones, where the consumption declines from $(k-1)\mu$ to $\mu$, leaving the last $r-k$ segments completely untouched. The exact worst case count yields that
  together  twins $T_1,\dots, T_r$ cover at least
\begin{equation}\label{bound}
\left[m-(k-1)-(r-1)\right]k\mu+2\binom k2\mu=(m-r+1)k\mu
 \end{equation}
 elements of $w$.

The R-H-S of \eqref{bound} is, given $m\to\infty$ as $n\to\infty$, very close to $mk\mu$, which, in turn, will be close to $n$, provided
 $\mu\sim N/k$. To get these relations, the authors of~\cite{APP} developed a regularity lemma for words. Here we are in a more comfortable situation as we are dealing with random words. Therefore, a simple application of Chernoff's bound yields the following result.

\begin{lemma}\label{interwining}
For $r\ge k\geq2$, with probability at least $1-O(n^{-k+1/3})$,
$$t^{(r)}(W_k(n))=n/r-O\left( n^{2/3}\sqrt{\log n}\right).$$
\end{lemma}

\proof Split the set of positions of $W_k(n)$ into $m:=n^{1/3}$ consecutive segments of length $N:=n^{2/3}$. For $j=1,\dots,m$ and $i=1,\dots,k$, let $X:=X_i^j$ be the number of elements $a_i$ in the $j$-th segment. Then $\E X=N/k$ and, by Chernoff's bound from~\cite{JLR}, ineq. (2.6), we have
\[
\PP(X\le(1-\eps)\E X)\le n^{-k},
\]
where $\eps=kn^{-1/3}\sqrt{2\log n}$.
As we only have $km=O(n^{1/3})$ random variables $X_i^j$, with probability at least $1-O(n^{-k+1/3})$, they all satisfy  the opposite bound, that is, for all $j=1,\dots,m$ and $i=1,\dots,k$,
$$X_i^j\ge \mu:=(1-\eps)n^{2/3}/k=n^{2/3}/k -O(n^{1/3}\sqrt{\log n}).$$
Upon applying \eqref{bound} with this $\mu$, we conclude that the number of elements uncovered by the $r$-twins is, indeed, at most
$n - (m-r+1)k\mu = O(n^{2/3}\sqrt{\log n})$.
\qed

\medskip

Notice that for $r=k=2$, the estimate for $t(W_2(n))$ in the lemma is slightly weaker than the one from \cite{HeHNT} mentioned in the introduction. However, this fact does not affect our results, as all we really need in the proofs (see Example \ref{375} and the proof of Theorem \ref{main2}) is an estimate $t^{(r)}(W_k(n))=n/r-o(1)$. Besides, it is not clear how to generalize the result from \cite{HeHNT} to other values of $k$ and $r$.

\subsection{Equivalence of models of random words}\label{equiv}
Guided by analogy to random graphs, we consider  two basic models of random words: the binomial and the fixed-letter-count model. Given positive integers $k$ and $n$, an alphabet $A=\{a_1,\dots,a_k\}$, and constants $0\le p_1,\dots,p_k\le 1$, where $p_1+\cdots+ p_k=1$, the \emph{binomial random word}
$W(n;p_1,\dots,p_k)$ is a sequence of independent random variables $(X_1,\dots,X_n)$, where for each $j=1,\dots,n$ and $i=1,\dots,k$ we have $\PP(X_j=a_i) = p_i$.
Here we are exclusively interested in the special, equiprobable instance $W_k(n):=W(n;1/k,\dots,1/k)$.

As a technical tool rather than an object of genuine interest we also define another model of a random word. Given integers $0\le M_1,\dots, M_k\le n$, where $M_1+\dots +M_k=n$, \emph{the fixed-letter-count random word} $W(n;M_1,\dots,M_k)$ is obtained by taking uniformly at random a permutation (with repetitions) of $n$ elements, among which, $i=1,\dots,k$, there are $M_i$ elements $a_i$. Thus, in the latter model, we restrict ourselves to words with prescribed numbers of each letter and every such word has the same probability $\binom n{M_1,\dots,M_k}^{-1}$ to be chosen.

The asymptotic equivalence between the two models goes smoothly one way (from fixed-letter-count to binomial), but to proceed the other way  the models lack monotonicity, so we are forced to recourse to an analog of Pittel's inequality (see, e.g., ineq.~(1.6) in~\cite{JLR}) which bring, however, some limitations.

Let $Q$ be a property (subset) of words of length $n$ over alphabet~$A$. Although more general statements can be easily proved, we restrict ourselves to the special case of the  model $W_k(n)$ and also to limiting probabilities equal to 1 only. The proofs follow those for random graphs (cf. Section 1.4 in~\cite{JLR}, in particular, proofs of Prop. 1.12 and of Pittel's inequality (1.6) therein) and rely on the law of total probability and the fact that the space of $W_k(n)$ conditioned on $M_i$'s being the numbers of occurrences of the elements $a_i\in A$, $i=1,\dots,k$, coincides with that of
$W(n;M_1,\dots,M_k)$.

\begin{prop}\label{M2p}
If for all $M_1,\dots,M_k$ such that $M_1+\cdots+M_k=n$ and $M_i=n/k+O(\sqrt n)$, $i=1,\dots,k$, $\PP(W(n;M_1,\dots,M_k)\in Q)\to1$ as $n\to\infty$, then $\PP(W_k(n)\in Q)\to1$ as $n\to\infty$.
\end{prop}

\proof Let $C$ be a large constant and define (for each $n$)
$${\mathcal M}_C=\{(M_1,\dots,M_k): M_1+\cdots+M_k=n\quad\mbox{and}\quad|M_i-n/k|\le C\sqrt n\}.$$
Let $(M^*_1,\dots,M^*_k)$ minimize $\PP(W(n;M_1,\dots,M_k)\in Q)$ over ${\mathcal M}_C$.
Finally, let $X_i$ be the number of occurrences of letter $a_i$ in $W_k(n)$. Note that each $X_i$ has binomial distribution with expectation $n/k$ and variance less than $n/k$.
Then, by the law of total probability
$$\PP(W_k(n)\in Q)\ge\PP(W(n;M^*_1,\dots,M^*_k)\in Q)\PP((X_1,\dots,X_k)\in{\mathcal M}_C).$$
By assumption, $\PP(W(n;M^*_1,\dots,M^*_k)\in Q)\to1$ as $n\to\infty$. By Chebyshev's inequality applied together with the union bound,
$$\PP((X_1,\dots,X_k)\not\in{\mathcal M}_C)\le k\frac{n/k}{C^2n}=C^{-2}.$$
Thus, $\liminf_{n\to\infty}\PP(W_k(n)\in Q)\ge1-C^{-2}$. As this is true for every $C$, letting $C\to\infty$ yields $\lim_{n\to\infty}\PP(W_k(n)\in Q)=1$. \qed

\begin{prop}\label{p2M}
If  $\PP(W_k(n)\in Q)=1-o(n^{-k/2})$ as $n\to\infty$, then for all $M_i=n/k+\omega_i$, where $|\omega_i|\le\sqrt{(n\log n)/(3k^2)}$ and $\sum_i\omega_i=0$, $\PP(W(n;M_1,\dots,M_k)\in Q)\to1$ as $n\to\infty$.
\end{prop}

\proof Fix $M_1,\dots,M_k$ as in the statement of the proposition. By the law of total probability, we obtain
\begin{align*}
\PP(W_k(n)\not\in Q)&=\sum_{M'_1,\dots,M'_k}\PP(W(n;M'_1,\dots,M'_k)\not\in Q)\binom n{M_1',\dots,M'_k}\frac1{k^n}\\&\ge \PP(W(n;M_1,\dots,M_k)\not\in Q)\binom n{M_1,\dots,M_k}\frac1{k^n},
\end{align*}
from which we get
$$\PP(W(n;M_1,\dots,M_k)\not\in Q)\le\frac{k^n}{\binom n{M_1,\dots, M_k}}\PP(W_k(n)\not\in Q).$$
It remains to estimate the ratio $\frac{k^n}{\binom n{M_1,\dots, M_k}}$. Using Stirling's formula several times, we get
\[
\frac{k^n}{\binom n{M_1,\dots, M_k}}=O( n^{(k-1)/2})\prod_{i=1}^k(1+k\omega_i/n)^{M_i}\\
\le O( n^{(k-1)/2})\exp\left\{\sum_{i=1}^{k}k\omega_iM_i/n\right\}.
\]
Now since $M_i=n/k+\omega_i$ and $\sum_i\omega_i=0$, we get
\[
\sum_{i=1}^{k}k\omega_iM_i/n
= \sum_{i=1}^{k}k\omega_i\left(\frac{1}{k}+\frac{\omega_i}{n}\right)
= \sum_{i=1}^{k}\frac{k\omega_i^2}{n}
\le \frac{\log n}{3}
\]
and so ${k^n}/{\binom n{M_1,\dots, M_k}} = O( n^{k/2 - 1/6})$,
which yields that $\PP(W_k(n)\not\in Q)=o(1)$. \qed

\subsection{Boosting Lemma}\label{boost}
Fix $2\le r\le k$ and observe that if for some $\lambda=\lambda(n)>0$ and all $n\ge n_0$, we have $t^{(r)}(n,k)\ge\lambda n$, then also, by dropping the least frequent letter, $t^{(r)}(N,k+1)\ge\frac{\lambda k}{k+1}N$, provided $N\ge\tfrac{k+1}k n_0$.
In this section we show that for random words this trivial bound can be improved if one considers the fixed-letter-count model.

 To get a similar result for the binomial model $W_k(n)$ one has to switch first to the fixed-letter-count model $W(n;M_1,\dots,M_{k})$ and then back to $W_{k+1}(N)$, the switches facilitated, respectively, by Propositions \ref{p2M} and \ref{M2p}.
The reason for switching  is that the fixed-letter-count model can be broken into two phases allowing the enlargement of the twins.

Indeed, let $M_1,\dots,M_{k+1}$ be given such that $\sum_{i=1}^{k+1}M_i=N$. We generate the random word $W(N;M_1,\dots,M_{k+1})$ by first permuting all $n:=M_1+\cdots+M_k$ letters from $A\setminus\{a_{k+1}\}$ (Phase 1). This can be done in precisely $\binom{n}{M_1,\dots,M_k}$ ways.
Then we throw in the $M_{k+1}$ letters $a_{k+1}$ which can go anywhere between the previously  distributed letters (Phase 2). This can be done, by the formula for the number of ways to allocate $M_{k+1}$ balls into $n+1$ bins, in $\binom{N}{M_{k+1}}$ ways. Note that the product of these two numbers is, indeed, $\binom N{M_1,\dots,M_{k+1}}$ and that the outcome of Phase 1 is precisely the random word $W(n;M_1,\dots,M_k)$.


\begin{lemma}[Boosting Lemma]\label{k2k+1}
For $2\le r\le k$ and $n$ sufficiently large, let a partition $n=M_1+\cdots+M_k$ into nonnegative integers be given. If $M_{k+1}=n/k+O(\sqrt n)$ and, for some $\lambda=\lambda(n)>0$,  a.a.s.~$t^{(r)}(W(n;M_1,\dots,M_k))\ge\lambda n$, then a.a.s.
$$t^{(r)}(W(N;M_1,\dots,M_k,M_{k+1}))\ge\left(1+\frac1{(k+1)^r-1}\right)\frac{\lambda k}{k+1}N(1-o(1)),$$
where $N=n+M_{k+1}$.
\end{lemma}
\noindent Note that the factor of $\frac{\lambda k}{k+1}N$ comes for free already after Phase 1, so that the actual improvement sits in the parentheses. Also, notice that
$$\frac{n}{N}=\frac{k}{k+1}\left(1+O\left(\frac1{\sqrt n}\right)\right).$$

In the proof of Lemma \ref{k2k+1} we will need a technical estimate on a ratio of binomial coefficients, the proof of which is deferred to  Appendix~\ref{sec:fact}.
\begin{fact}\label{fact:ratio}
Let $M = \Theta(N)$ and $\ell^2 = o(N)$. Then,
\[
\frac{\binom{N-\ell}{M-\ell}}{\binom{N}{M}} = \left( \frac{M}{N} \right)^\ell \left(1+O\left(\frac{\ell^2}{N}\right) \right).
\]
\end{fact}

\begin{proof}[Proof of Lemma~\ref{k2k+1}]
We generate $W(N;M_1,\dots,M_{k+1})$ in two rounds as described prior to the statement of the lemma. Let $W'$ be the outcome of round one, that is, an instance of $W(n;M_1,\dots,M_k)$. Further, let $Q'$ be the event that $t^{(r)}(W(n;M_1,\dots,M_k))\ge\lambda n$ and $Q$ -- the ultimate event that
\[
t^{(r)}(W(N;M_1,\dots,M_k,M_{k+1}))\ge\left(1+\frac1{(k+1)^r-1}\right)\frac{\lambda k}{k+1}N(1-o(1)).\]
By the law of total probability,
\begin{equation}\label{PQ}
\PP(Q)=\sum_{W'}\PP(Q|W')\PP(W')=\sum_{W'\in Q'}\PP(Q|W')\PP(W')+o(1).
\end{equation}
We now focus on $\PP(Q|W')$ with $W'\in Q'$, that is, fixing an instance $W'$ of round one with $t^{(r)}(W)\ge\lambda n$, we are going to thoroughly  investigate the outcome of the second round.
Fix $W'=w_1\cdots w_n$ and $r$-twins $T_1,\dots,T_r$ therein of length $|T_1|=\cdots=|T_r|=\lambda n$ (we may assume that $\lambda n$ is an integer).  We treat the $n+1$ spaces before, between, and after the letters of $W'$ as bins and the $M_{k+1}$ letters $a_{k+1}$ as balls. Precisely, bin $b_0$ is in front of $w_1$, for  $i=1,\dots,n-1$, bin $b_i$ lies between $w_i$ and $w_{i+1}$, and bin $b_n$ is to the right of $w_n$.

We group the bins lying immediately  to the right of the elements of the $r$-twins $T_1,\dots,T_r$ into \emph{$r$-tuples of bins} denoted by $R_1,\dots, R_{\lambda n}$. Formally, if $T_j=w_{i_1}^{(j)}\cdots w_{i_{\lambda n}}^{(j)}$, $j=1,\dots,r$, then the $r$-tuple of bins $R_\ell$ consists of the bins $b_{i_\ell}^{(1)},\dots, b_{i_\ell}^{(r)}$, for $\ell=1,\dots,\lambda n$.

\begin{exmp} Let $k=r=3$, $n=27$, and
$$W'=w_1\cdots w_{27}= b\; \colorbox{cyan}{$a$}\; \colorbox{cyan}{$a$}\; c\; \colorbox{Lavender}{$a$}\; c\; b\; \colorbox{cyan}{$b$}\; \colorbox{Lavender}{$a$}\; \colorbox{Lavender}{$b$}\; \colorbox{cyan}{$c$}\; \colorbox{green}{$a$}\; \colorbox{Lavender}{$c$}\; a\; c\; \colorbox{green}{$a$}\; b\; b\; \colorbox{cyan}{$c$}\; \colorbox{green}{$b$}\; a\; \colorbox{green}{$c$}\; \colorbox{Lavender}{$c$}\; \colorbox{green}{$c$}\; b\; a\; b.$$
There are 3-twins of length 5 here, each forming the word $aabcc$, namely, $T_1=w_2w_3w_8w_{11}w_{19}$, $T_2=w_5w_9w_{10}w_{13}w_{23}$, and $T_3=w_{12}w_{16}w_{20}w_{22}w_{24}$. For instance, consider the first triple of bins, $R_1=\{b_2,b_5,b_{12}\}$. If the letter $d$ is inserted into each of these three bins,  a longer word
$$b\; \colorbox{cyan}{$a$}\; \underline{d}\; \colorbox{cyan}{$a$}\; c\; \colorbox{Lavender}{$a$}\; \underline{d}\; c\; b\; \colorbox{cyan}{$b$}\; \colorbox{Lavender}{$a$}\; \colorbox{Lavender}{$b$}\; \colorbox{cyan}{$c$}\; \colorbox{green}{$a$}\; \underline{d}\; \colorbox{Lavender}{$c$}\; a\; c\; \colorbox{green}{$a$}\; b\; b\; \colorbox{cyan}{$c$}\; \colorbox{green}{$b$}\; a\; \colorbox{green}{$c$}\; \colorbox{Lavender}{$c$}\; \colorbox{green}{$c$}\; b\; a\; b,$$
is obtained, and  the length of the twins, each forming now the word $adabcc$, increases by one.
\end{exmp}

Clearly, if each bin from an $r$-tuple  receives $s$ balls (read:~$s$ letters $a_{k+1}$), then the length of the twins can be extended by $s$.
Since we would like to  utilize several $r$-tuples of bins simultaneously, we categorize them with respect to the minimum number of balls.

For $1\le s \le \log n$, let $X_s$ be the number of $r$-tuples of bins with at least $s$ balls in each bin and exactly $s$ balls in some bin (i.e. there is a bin with exactly $s$ balls). We call each such $r$-tuple of bins an \emph{$s$-provider}. For each $\ell=1,\dots,\lambda n$, let $I_\ell$ be the indicator random variable such that $I_\ell=1$ if $R_\ell$ is an $s$-provider and 0 otherwise.
Hence, $X_s=\sum_{\ell=1}^{\lambda n}I_\ell$.

The event $\{I_\ell=1\}$ is the set difference of two events: that  each bin in $R_\ell$ has at least $s$ balls and that each bin in $R_\ell$ has at least $s+1$ balls. Thus, setting $M:=M_{k+1}$,
\[
\PP(I_i=1)=\frac{\binom{N-rs}{M-rs}}{\binom{N}{M}}-\frac{\binom{N-rs-r}{M-rs-r}}{\binom{N}{M}}
\]
and Fact~\ref{fact:ratio} yields
\begin{align*}
\PP(I_i=1) &= \left( \frac{M}{N} \right)^{rs} \left(1-\left( \frac{M}{N} \right)^{r} \right) \left(1+O\left(\frac{s^2}{N}\right) \right)\\
&= \left( \frac{1}{k+1} \right)^{rs} \left(1-\left( \frac{1}{k+1} \right)^{r} \right) \left(1+O\left(\frac1{\sqrt n}\right)\right).
\end{align*}
Thus, setting $\kappa=(k+1)^{-r}$, we have $\E X_s=\lambda n(1-\kappa)\kappa^s \left(1+O\left(1/{\sqrt n}\right)\right)$.

Our goal is to show a.a.s.~simultaneous concentration of each $X_s$, $s\le \log n$, near its expectation.
We are going to use Chebyshev's inequality (followed by the union bound over all $s$)
$$\PP(|X_s-\E X_s|\ge\gamma\E X_s)\le\frac{Var X_s}{(\gamma\E X_s)^2}$$
with $\gamma = 1/\log n$.
To facilitate the future use of the union bound, we need to show that $Var X_s=o((\E X_s)^2/\log^3 n)$, which will imply that $\frac{Var X_s}{(\gamma\E X_s)^2} = o(1/\log n)$.
To this end we write
$$Var X_s=\E(X_s(X_s-1))+\E X_s-(\E X_s)^2.$$
Note that
$\{I_{\ell_1}=I_{\ell_2}=1\}=A\setminus(B_1\cup B_2)$, where $A$ is the event that all $2r$ bins  in $R_{\ell_1}$ and $R_{\ell_2}$ each contains at least $s$ balls, while $B_i$, $i=1,2$, is the event that all bins in $R_{\ell_i}$ each contains at least $s$ balls and those in $R_{\ell_{3-i}}$ -- each at least $s+1$ balls. Thus,
\begin{align*}
\PP(I_{\ell_1}=I_{\ell_2}=1)&=\PP(A)-\PP(B_1)-\PP(B_2)+\PP(B_1\cap B_2)\\
&=\frac{\binom{N-2rs}{M-2rs}}{\binom{N}{M}}-2\frac{\binom{N-2rs-r}{M-2rs-r}}{\binom{N}{M}}+\frac{\binom{N-2rs-2r}{M-2rs-2r}}{\binom{N}{M}}
\end{align*}
which by Fact~\ref{fact:ratio} is equal to
\[
\left( \kappa^{2s}-2\kappa^{2s+1}+\kappa^{2s+2} \right) \left(1+O\left(\frac1{\sqrt n}\right)\right)
=\left((1-\kappa)\kappa^s\right)^2 \left(1+O\left(\frac1{\sqrt n}\right)\right).
\]
Thus,
\begin{align*}
Var X_s
&= \lambda n(\lambda n-1) \left((1-\kappa)\kappa^s\right)^2 \left(1+O\left(\frac1{\sqrt n}\right)\right)\\
&\qquad + \lambda n(1-\kappa)\kappa^s \left(1+O\left(\frac1{\sqrt n}\right)\right)
- \left( \lambda n(1-\kappa)\kappa^s \left(1+O\left(\frac1{\sqrt n}\right)\right)\right)^2\\
&=  \left(\lambda n(1-\kappa)\kappa^s\right)^2 +  O(n^{3/2})\\
&\qquad + O(n)
- \left(\lambda n(1-\kappa)\kappa^s\right)^2 -  O(n^{3/2})
=  O(n^{3/2}) = o((\E X_s)^2/\log^3 n),
\end{align*}
with a big margin.

Hence,
$$\sum_{s=1}^{\lfloor\log n\rfloor}\PP(|X_s-\E X_s|\ge\gamma\E X_s)=o(1),$$
and, in particular, a.a.s.,  for all $1\le s\le \log n$,
\begin{align*}
X_s &\ge \E X_s\left(1-\frac1{\log n}\right)\\
& = \lambda n(1-\kappa)\kappa^s \left(1+O\left(\frac1{\sqrt n}\right)\right)\left(1-\frac1{\log n}\right)
= \lambda n(1-\kappa)\kappa^s \left(1-O\left(\frac1{\log n}\right)\right).
\end{align*}
As each $X_s$ contributes $s$ towards an enlargement of the twins $T_1,\dots,T_r$, we need to calculate $\sum_{s=1}^{\lfloor\log n\rfloor} s X_s$.
Since for every positive integer $p$
\[
\sum_{s=1}^{p}s\kappa^s = \sum_{s=1}^{\infty}s\kappa^s - \sum_{s=p+1}^{\infty}s\kappa^s
= \frac{\kappa}{(1-\kappa)^2} - \frac{(1+(1-k)p)\kappa^{p+1}}{(1-\kappa)^2},
\]
we get
\begin{align*}
\sum_{s=1}^{\lfloor\log n\rfloor} s X_s
&\ge \lambda n(1-\kappa) \left(1-O\left(\frac1{\log n}\right)\right) \sum_{s=1}^{\lfloor\log n\rfloor} s\kappa^s \\
&= \lambda n(1-\kappa) \left(1-O\left(\frac1{\log n}\right)\right) \left( \frac{\kappa}{(1-\kappa)^2} - \Theta(\kappa^{\log n }) \right)\\
&=  \frac{\lambda\kappa}{1-\kappa} n (1-o(1)) =  \frac{\lambda}{(1+k)^r-1} n (1-o(1)).
\end{align*}
Hence, still conditioning on $W'$, a.a.s., there are in $W(N;M_1,\dots,M_k,M_{k+1})$ twins of length at least
\[
\lambda n  + \frac{\lambda}{(1+k)^r-1} n (1-o(1))
= \left(1+\frac{1}{(1+k)^r-1}\right) \lambda n (1-o(1)).
\]
As $n = \frac{k}{k+1}N(1+o(1))$, this means that $\PP(Q|W')=1-o(1)$ for every $W'\in Q'$, and, by~\eqref{PQ}, the lemma follows.
\end{proof}

\begin{rem}\label{simpler} By considering only one random variable $X$ counting $r$-tuples of bins with no bin empty, the proof becomes a bit simpler, but the result is slightly weaker: a.a.s.
$$t^{(r)}(W(N;M_1,\dots,M_k,M_{k+1}))\ge\left(1+\frac1{(k+1)^r}\right)\frac{\lambda k}{k+1}N.$$
We get rid of $o(1)$, but lose $-1$ in the denominator, so overall the bound is lower, though for large $k$ the difference is insignificant.
\end{rem}

\subsection{Applications}\label{apple} In this subsection we present applications of the Boosting Lemma (Lemma \ref{k2k+1}), most importantly, in the proofs of Theorems \ref{main1} and \ref{main2}. Each time the scenario is the same: we pick an existing lower bound on  the length of twins in $W_{k'}(n)$ (e.g., the bound in Theorem~\ref{thm:0.4} above or any deterministic result from~\cite{APP} or~\cite{BukhZ}), translate it via Proposition~\ref{p2M} to the fixed-letter-count model and then apply (iteratively) the Boosting Lemma to get a bound, still in the fixed-letter-count model, for the targeted value of $k>k'$.
At the end we  go back to the binomial model $W_k(n)$ via Proposition \ref{M2p}. For this we need to have our estimate valid for all $M_i$, $i=1,\dots,k$, satisfying the assumptions of Proposition~\ref{M2p}.
This means, however, that  also the input bound has to be valid for a corresponding range of $M_i$, $i=1,\dots,k'$.

Let us now analyze what it really means. Fix $k$ and let $M_i=n/k+w_i$, $i=1,\dots,k$, where $|w_i|=O(\sqrt n)$ and $\sum_{i=1}^kw_i=0$. How these assumptions alter when we drop $M_k$? Let $n_{k-1}=M_1+\cdots+M_{k-1}$. Then, for each $i=1,\dots,k-1$,
$$n_{k-1}=\frac{k-1}kn+\sum_{i=1}^{k-1}w_i=(k-1)M_i-(k-1)w_i-w_k,$$
thus
$$M_i=\frac{n_{k-1}}{k-1}+w_i+\frac{w_k}{k-1}=\frac{n_{k-1}}{k-1}+w_i^{(k-1)},$$
where $w_i^{(k-1)}=w_i+\frac{w_k}{k-1}$. Note that $\sum_{i=1}^{k-1}w_i^{(k-1)}=0$ as it should. We may iterate this relation all the way down to $k'$ (where we want to begin the process of applying Lemma \ref{k2k+1}), obtaining for $j=k-1,\dots,k'$ and $i=1,\dots,k'$,
with $n_j=M_1+\cdots+M_{j}$,

\begin{equation}\label{MM}
M_i=\frac{n_j}{k}+w_i^{(j)},
\end{equation}
where
\begin{equation}\label{ww}
w_i^{(j)}=w_i+\frac1{j}\sum_{q=j+1}^kw_q.
\end{equation}
Observe that $w_i^{(k')}=O(\sqrt n)$, so we stay within the range required in Proposition \ref{p2M}. Below we illustrate how to apply the Boosting Lemma.
\begin{exmp}\label{375}
Let $r=2$ and $k=3$. We know, either by our Lemma \ref{interwining} or by the result in~\cite{APP},  that a.a.s.~a random binary word $W_2(n)$ contains twins of length $\left(\tfrac12-o(1)\right)n$. We would like to deduce from this, via Lemma \ref{k2k+1}, a lower bound on the length of twins in the random ternary word $W_3(n)$ (so, we take $k'=2$ here).

Fix $M_i=n/3+w_i$, $i=1,2,3$, where $|w_i|=O(\sqrt n)$ and $w_1+w_2+w_3=0$. Suppressing $M_3$, we get, with $n'=n-M_3$, an instance of $W(n';M_1,M_2)$ which satisfies the assumptions of Proposition \ref{p2M}, that is, for $i=1,2$, we have $M_i=n'/2+w_i^{(2)}$, where $w_i^{(2)}=w_i+\frac{w_3}{2}=O(\sqrt n)$. Thus, we may conclude that a.a.s.~$t^{(2)}(W(n';M_1,M_2))\ge\lambda n'$ with $\lambda=\tfrac12-o(1)$. In turn, by Lemma \ref{k2k+1}, a.a.s.
$$t^{(2)}(W(n;M_1,M_2,M_3))\ge\left(1+\frac18\right)\frac{2\lambda}3 n(1-o(1))\ge 0.375n(1-o(1)).$$
Since this is true for all choices of $M_i$ as above, by Proposition \ref{M2p}, we finally get that a.a.s.~$t^{(2)}(W_3(n))\ge0.375n(1-o(1))$. This is much less than the bound in Theorem~\ref{thm:0.4}, so the result has some value only for computer-skeptical readers. On the other hand, it is still better than the bound $t^{(2)}(W_3(n))\ge0.34n(1-o(1))$ in \eqref{Bukh-Zhou t_k(n)1}, though the latter holds for \emph{all} ternary words.

Continuing with this example, let us iterate applications of Lemma \ref{k2k+1} till, say, $k=10$.
Skipping details, we see that the obtained bound is a.a.s.
\begin{align*}
t^{(2)}(W_{10}(n))&\ge\frac98\cdot\frac{16}{15}\cdot\frac{25}{24}\cdot\frac{36}{35}\cdot\frac{49}{48}\cdot\frac{64}{63}\cdot\frac{81}{80}\cdot\frac{100}{99}\cdot\frac{n}{10}(1+o(1))\\
&=\frac{15}{11}\cdot\frac{n}{10}(1-o(1))=1.\overline{36}\cdot \frac{n}{10}(1-o(1)).
\end{align*}
It is perhaps interesting to compare the above bound with $1.337(n/10)$ -- one obtained by the simpler and weaker version of Lemma \ref{k2k+1} mentioned in Remark \ref{simpler}.
\end{exmp}

\begin{proof}[Proof of Theorem \ref{main1}]
As a starting point we take our computer-assisted Theorem~\ref{thm:0.4}, so we set $k'=3$. Preparing for the final transition to the binomial model, fix any $M_i=n/k+w_i$, $i=1,\dots,k$, where $|w_i|=O(\sqrt n)$ and $\sum_{i=1}^kw_i=0$. By \eqref{MM} with $j=3$, we then have $M_i=n_3/3+w_i^{(3)}$, where $n_3=M_1+M_2+M_3$ and  $w_i^{(3)}=O(\sqrt n)$, $i=1,2,3$, by \eqref{ww}.

By Theorem \ref{thm:0.4}, $t(W_3(n))\ge0.411n$ with probability at least $1-e^{-\Omega(n/(\log n))}=1-o(n^{-k/2})$. Thus, by Proposition \ref{p2M}, a.a.s.~$t(W(n_3;M_1,M_2,M_3))\ge0.411n$, and, in turn, by Boosting Lemma (Lemma \ref{k2k+1}) with $\lambda=0.411$, a.a.s.~$t(W(n_4;M_1,M_2,M_3,M_4))\ge\frac{16}{15}\cdot 1.233\cdot \frac{n_4}{4}$. We iterate this transition (or use  induction) until reaching the random word $W(n;M_1,\dots,M_k)$.
Observe that
\[
\prod_{j=4}^k\frac{j^2}{j^2-1}
= \prod_{j=4}^k\frac{j^2}{(j-1)(j+1)}
= \frac{4^2}{3\cdot 5} \cdot \frac{5^2}{4\cdot 6} \cdot \frac{6^2}{5\cdot 7} \cdots \frac{k^2}{(k-1)(k+1)} = \frac{4k}{3(k+1)}.
\]
Hence, we get, a.a.s.,
$$t^{(2)}(W(n;M_1,\dots,M_k))\ge \prod_{j=4}^k\frac{j^2}{j^2-1} \cdot (3\rho_{14}) \cdot\frac{n}{k} (1-o(1))
	\ge \frac{1.64k}{k+1}  \cdot\frac{n}{k}.$$
(We drop $0.004$ to ``kill'' the error term $1-o(1)$.)

As this is true for all choices of $M_i=n/k+w_i$, we are in position to apply Proposition~\ref{M2p} and conclude that a.a.s.
$t^{(2)}(W_k(n))\ge \frac{1.64k}{k+1}  \cdot\frac{n}{k}.$
\end{proof}

\begin{proof}[Proof of Theorem \ref{main2}] This proof is very similar except that we now take $k'=r$ and launch off with the bound $t^{(r)}(W_{k'}(n))=n/r-o(n)$ from  Lemma~\ref{interwining} (or its deterministic counterpart from~\cite{APP}). Other minor differences are that for $r\ge3$ there is no nice formula for the product $\Pi_{r,k}=\prod_{j=r+1}^k\frac{j^r}{j^r-1}$ and, unlike before, we cannot get rid of the error term $1-o(1)$. We omit the details.
\end{proof}

\section{Concluding remarks}\label{Section Concluding}
Let us conclude the paper with some open questions and suggestions for future studies. The first problem naturally concerns twins in ternary words.
\begin{prob} Does $t(n,3)=n/2-o(n)$?
\end{prob}
Even if the answer is negative it can still be true that \emph{almost all} ternary words contain twins of such length, that is, a.a.s. $t(W_3(n))=n/2-o(n)$. It is known, as demonstrated in~\cite{BukhZ}, that a.a.s. $t(W_4(n))\leq0.4932n$ which makes a similar statement for random \emph{quaternary} words false. In general, it is not clear how close  $t(W_k(n))$ and $t(n,k)$ are from each other.

\begin{prob}
	Does $t(W_k(n))=t(n,k)+o(n)$ hold a.a.s.~for all $k\geq 2$?
\end{prob}
\noindent By the result of Axenovich, Person, and Puzynina~\cite{APP} we know that this is the case for $k=2$.

One may also consider more restricted notions of twins in words reflecting their placement in the word. For instance, it may happen that twins occupy together a connected segment  forming a \emph{shuffle square}, as in the example below:
$$a\;b\;\colorbox{cyan}{c}\;\colorbox{cyan}{a}\;\colorbox{Lavender}{c}\;\colorbox{cyan}{b}\;\colorbox{Lavender}{a}\;\colorbox{Lavender}{b}\;c\;b\;a\;c.$$
As proved recently by Bulteau, Jug\'{e}, and Vialette \cite{BulteauJV}, there exist arbitrarily long words over a $6$-letter alphabet containing no shuffle squares. It is not known, however, if the size of the alphabet in this result is optimal and, more generally, what is the maximum length of a shuffle square guaranteed to be present in \emph{every} (long)  $k$-letter word, $k=2,\dots,5$.

Here, we formulate this question for random words. 

\begin{prob}
	What is the expected maximum length of a shuffle square in a random word $W_k(n)$? 	
\end{prob}

This question sounds particularly intriguing in the light of a recent conjecture by He, Huang, Nam, and Thaper \cite{HeHNT} (mentioned already in Section \ref{Section Twins History}), that almost every binary word (with even numbers of ones and zeros) is a shuffle square.  For more on counting shuffle squares and their variants see, e.g., \cite{HenshallRS}.  Similar problems can also be considered for \emph{shuffle cubes}, or more generally, for arbitrary \emph{shuffle $r$-powers}, in analogy to general $r$-twins.

An even more restricted version of twins is obtained when each twin in a shuffle square occupies itself a connected segment, like in the example below:
$$a\;b\;\colorbox{cyan}{c}\;\colorbox{cyan}{a}\;\colorbox{cyan}{b}\;\colorbox{Lavender}{c}\;\colorbox{Lavender}{a}\;\colorbox{Lavender}{b}\;c\;b\;a\;c.$$
This basic structure is well-known and widely studied in combinatorics on words under the name of a \emph{square} or a \emph{repetition} (see \cite{Lothaire1}, \cite{RampersadShallit}). By the famous result of Thue~\cite{Thue} we know that there exist ternary words of any length with no squares altogether. In the binary case, it is known that there exist arbitrarily long words avoiding squares of length greater than $2$, as proved by Fraenkel and Simpson~\cite{FraenkelSimpson}. However, not much is known about squares in random words.

\begin{prob}
	What is the expected maximum length of a square in a random word $W_k(n)$?
\end{prob}


\appendix

\section{Maple code}\label{sec:maple}

Here we present a simple program written in \emph{Maple} that counts the number of ternary words of length $s$ that contain twins of length at least $t$. This easily allows one to calculate the number of ternary words of length $s$ with twins of length $t$ but not $t+1$ (defined as $\lambda_t$ in the proof of Theorem~\ref{thm:0.4}). Indeed, $\lambda_t$ is equal to the difference between the number of words having twins of length at least $t$ minus the one with twins of length at least $t+1$.

For the sake of clarity, we did not attempt to optimize our program.

{\fontfamily{pcr}\selectfont
\begin{lstlisting}[language=Maple, caption={The main procedure.}, label={code:procedure}, basicstyle=\footnotesize]
with(combinat): # Load combinatorial functions

##########################################################
# This procedure returns the number of ternary words
# of length s that contain twins of length at least t
##########################################################
count_twins := proc(s, t)
	local P_list, P, sizeP, K, sizeK, L, sizeL, L1, L2, counter, p, k, l, i:
	
	# Generate all ternary words of length s
	P_list := [seq(1, i = 1 .. s), seq(2, i = 1 .. s), seq(3, i = 1 .. s)]:
	P := permute(P_list, s):
	sizeP := nops(P):
	
	# Generate all 2t-subsets of {1,...,s}
	K := choose([seq(i, i = 1 .. seq_length)], 2*t):
	sizeK := nops(K):
	
	# Partition {1,...,2t} into two sets, each of size t
	L := setpartition([seq(i, i = 1 .. 2*t)], t):
	sizeL := nops(L):
	
	counter := 0: # Number of words with twins of length at least t
	
	for p to sizeP # Over all ternary words
	do
		for k to sizeK # Over all 2t subsets of {1,...,s}
		do
			for l to sizeL # Over all partitions of {1,...,2t}
			do
				L1 := L[l][1]: # Indices for the first possible twin
				L2 := L[l][2]: # Indices for the second possible twin
				
				if evalb(P[p][K[k][L1]] = P[p][K[k][L2]]) # Twins?
				then
					# Twins found
					counter := counter + 1:
					
					# Move to the next word in P
					l := sizeL: # Break the loop over l
					k := sizeK: # Break the loop over k
				end if:
			end do: # Over l
		end do: # Over k
	end do: # Over p
	counter: # Return the value of the variable counter
end proc:
\end{lstlisting}
}

In Listing~\ref{code:call}  we show how the procedure from Listing~\ref{code:procedure} can be used yielding results summarized in Table~\ref{table:lambda}.

{\fontfamily{pcr}\selectfont
\begin{lstlisting}[language=Maple, caption={Calling the main procedure.}, label={code:call}, basicstyle=\footnotesize]
seq_length := 6: # Length of the ternary words
min_twin_length := 1: # Min length of possible twins
max_twin_length := floor(seq_length/2): # Max length of possible twins

for i from max_twin_length by -1 to min_twin_length do
	c[i] := count_twins(seq_length, i): # Number of twins of length >= i
	if i = max_twin_length
	then
		d[i] := c[i]: # Since there are no twins of length max_twin_length+1, c[max_twin_length] stores the number of words with the longest twins of length max_twin_length
	else
		d[i] := c[i] - c[i + 1]:
	end if:
	printf("Number of words with the longest twins of length %d = %d\n", i, d[i]):
end do:
\end{lstlisting}
}

\section{Proof of Fact \ref{fact:ratio}}\label{sec:fact}

\begin{proof}[Proof of Fact \ref{fact:ratio}]
We are going to use the following form of Stirling's formula:
\[
N! = \sqrt{2\pi N}\left( \frac{N}{e}\right)^N \left(1+O\left(\frac{1}{N}\right) \right).
\]
Hence,
\[
\frac{N!}{(N-\ell)!} = \sqrt{\frac{N}{N-\ell}} \cdot \frac{N^N}{(N-\ell)^{N-\ell}} \cdot e^{-\ell} \left(1+O\left(\frac{1}{N}\right) \right).
\]
Since
\[
\sqrt{\frac{N}{N-\ell}} = \left(1 +\frac{\ell}{N-\ell} \right)^{1/2} = 1+O\left(\frac{\ell}{N}\right),
\]
we get
\[
\frac{N!}{(N-\ell)!} = \frac{N^N}{(N-\ell)^{N-\ell}} \cdot e^{-\ell} \left(1+O\left(\frac{\ell}{N}\right) \right).
\]
Consequently, as
\[
\frac{\binom{N-\ell}{M-\ell}}{\binom{N}{M}} = \frac{{M!}/{(M-\ell)!}}{{N!}/{(N-\ell)!}},
\]
we obtain
\[
\frac{\binom{N-\ell}{M-\ell}}{\binom{N}{M}}
=\frac{{M^M}/{(M-\ell)^{M-\ell}}}{{N^N}/{(N-\ell)^{N-\ell}}} \left(1+O\left(\frac{\ell}{N}\right) \right).
\]
Further, since $1+x = e^{x+O(x^2)}$ whenever $x\to 0$,
\[
\frac{{M^M}/{(M-\ell)^{M-\ell}}}{{N^N}/{(N-\ell)^{N-\ell}}}
= \left( \frac{M}{N} \right)^\ell \frac{\left(1-\frac{\ell}{N} \right)^{N-\ell}}{\left(1-\frac{\ell}{M} \right)^{M-\ell}}
= \left( \frac{M}{N} \right)^\ell \frac{e^{-\ell(N-\ell)/N +O(\ell^2/N)}}{e^{-\ell(M-\ell)/M +O(\ell^2/M)}}
= \left( \frac{M}{N} \right)^\ell e^{O(\ell^2/N)},
\]
where the last equality follows, because $-\ell(N-\ell)/N + \ell(M-\ell)/M = \ell^2(M-N)/(MN)$ and $M = \Theta(N)$. Finally, as by assumption $\ell^2 = o(N)$, we have $e^{O(\ell^2/N)} = 1 + O(\ell^2/N)$ and so
\[
\frac{\binom{N-\ell}{M-\ell}}{\binom{N}{M}}
=\left( \frac{M}{N} \right)^\ell \left( 1 + O\left(\frac{\ell^2}{N}\right) \right)\left(1+O\left(\frac{\ell}{N}\right) \right)
=\left( \frac{M}{N} \right)^\ell \left( 1 + O\left(\frac{\ell^2}{N}\right) \right),
\]
as required.
\end{proof}

\end{document}